\documentclass{article}%
\usepackage{amsfonts}
\usepackage{amsmath}
\usepackage{amssymb}
\usepackage{graphicx}%
\newtheorem{theorem}{Theorem}

\newenvironment{proof}[1][Proof]{\noindent\textbf{#1.} }{\ \rule{0.5em}{0.5em}}
\begin{document}

\title{Neuberg cubics over finite fields}
\author{N.\ J. Wildberger\\School of Mathematics and Statistics\\UNSW Sydney 2052 Australia}
\maketitle
\date{}

\begin{abstract}
The framework of universal geometry allows us to consider metrical properties
of affine views of elliptic curves, even over finite fields. We show how the
Neuberg cubic of triangle geometry extends to the finite field situation and
provides interesting potential invariants for elliptic curves, focussing on an
explicit example over $\mathbb{F}_{23}$. We also prove that tangent conics for
a Weierstrass cubic are identical or disjoint.

\end{abstract}

\section*{Metrical views of cubics}

This paper looks at the connection between modern Euclidean triangle geometry
and the arithmetic of elliptic curves over finite fields using the framework
of universal geometry (see \cite{Wild}), a metrical view of algebraic geometry
based on the algebraic notions of \textit{quadrance }and \textit{spread}
rather than \textit{distance} and \textit{angle.} A good part of triangle
geometry appears to extend to finite fields. In particular, the Neuberg cubic
provides a rich organizational structure for many triangle centers and
associated lines through the group law and related Desmic (linking) structure,
even in a finite field. It and other triangle cubics have the potential to be
useful geometrical tools for understanding elliptic curves. See
\cite{Cundy-Parry 1}, \cite{Cundy-Parry 2}, \cite{Cundy-Parry 3}, \cite{EG},
\cite{Kimberling1} and \cite{Kimberling2} for background on triangle cubics.

For triangle geometers, finite fields hold potential applications to
cryptography and also provide a laboratory for exploration that in many ways
is more pleasant than the decimal numbers. A price to be paid, however, is
that the usual tri-linear coordinate framework needs to be replaced by
Cartesian or barycentric coordinates.

We begin with a brief review of the relevant notions from rational
trigonometry, which allows the set-up of metrical algebraic geometry. Then we
discuss the Neuberg cubic of a triangle and related centers, illustrated in a
particular example over $\mathbb{F}_{23}$. We also prove that for affine
cubics in Weierstrass form the tangent conics are all disjoint provided $-3$
is not a square in the field.

It should be noted that there is also a projective version of universal
geometry (see \cite{Wild2}), but here we stick to the affine situation.

\section*{Laws of Rational Trigonometry}

Fix a finite field $\mathbb{F}$ not of characteristic two, whose elements are
called \textbf{numbers}. A \textbf{point} $A$ is an ordered pair $\left[
x,y\right]  $ of numbers, that is an element of $\mathbb{F}^{2}.$ The
\textbf{quadrance\ }$Q\left(  A_{1},A_{2}\right)  $ between points
$A_{1}\equiv\left[  x_{1},y_{1}\right]  $ and $A_{2}\equiv\left[  x_{2}%
,y_{2}\right]  $ is the number
\[
Q\left(  A_{1},A_{2}\right)  \equiv\left(  x_{2}-x_{1}\right)  ^{2}+\left(
y_{2}-y_{1}\right)  ^{2}.
\]

A \textbf{line} $l$ is an ordered proportion $\left\langle a:b:c\right\rangle
$, where $a$ and $b$ are not both zero. This represents the equation
$ax+by+c=0.$ Such a line is \textbf{null} precisely when
\[
a^{2}+b^{2}=0.
\]
Null lines occur precisely when $-1$ is a square. For distinct points
$A_{1}=\left[  x_{1},y_{1}\right]  $ and $A_{2}=\left[  x_{2},y_{2}\right]  $
the line passing through them both is
\[
A_{1}A_{2}=\left\langle y_{1}-y_{2}:x_{2}-x_{1}:x_{1}y_{2}-x_{2}%
y_{1}\right\rangle .
\]
Two lines $l_{1}\equiv\left\langle a_{1}:b_{1}:c_{1}\right\rangle $ and
$l_{2}\equiv\left\langle a_{2}:b_{2}:c_{2}\right\rangle $ are
\textbf{perpendicular} precisely when
\[
a_{1}a_{2}+b_{1}b_{2}=0.
\]
For any fixed line $l$ and any point $A,$ there is a unique line $n$ passing
through $A$ and perpendicular to $l,$ called the \textbf{altitude} from $A$ to
$l.$ If $l$ is a non-null line then the altitude $n$ meets $l$ at a unique
point $F,$ called the \textbf{foot }of the altitude. In this case we may
define the \textbf{reflection of }$A$\textbf{\ in }$l$ to be the point
$\sigma_{l}\left(  A\right)  $ such that $F$ is the midpoint of the side
$\overline{A\sigma_{l}\left(  A\right)  }$. If $m$ is another line, then the
\textbf{reflection of }$m$\textbf{\ in }$l$ is the line $\Sigma_{l}\left(
m\right)  $ with the property that the reflection in $l$ of any point $A$ on
$m$ lies on $\Sigma_{l}\left(  m\right)  $.

The \textbf{spread}\textit{\ }$s\left(  l_{1},l_{2}\right)  $ between non-null
lines $l_{1}\equiv\left\langle a_{1}:b_{1}:c_{1}\right\rangle $ and
$l_{2}\equiv\left\langle a_{2}:b_{2}:c_{2}\right\rangle $ is the number%
\[
s\left(  l_{1},l_{2}\right)  \equiv\frac{\left(  a_{1}b_{2}-a_{2}b_{1}\right)
^{2}}{\left(  a_{1}^{2}+b_{1}^{2}\right)  \left(  a_{2}^{2}+b_{2}^{2}\right)
}=1-\frac{\left(  a_{1}a_{2}+b_{1}b_{2}\right)  ^{2}}{\left(  a_{1}^{2}%
+b_{1}^{2}\right)  \left(  a_{2}^{2}+b_{2}^{2}\right)  }.
\]
This number $s=s\left(  l_{1},l_{2}\right)  $ is $0$ precisely when the lines
are parallel, and $1$ precisely when the lines are perpendicular. It has the
property that $s\left(  1-s\right)  $ is a square in the field, and every such
\textbf{spread number} $s$ can be shown to be the spread between some two lines.

The spread between lines may alternatively be expressed as a ratio of
quadrances: if $l_{1}$ and $l_{2}$ intersect at a point $A,$ choose any other
point $B$ on $l_{1}$, and let $C$ on $l_{2}$ be the foot of the altitude line
from $B$ to $l_{2},$ then
\[
s\left(  l_{1},l_{2}\right)  =\frac{Q\left(  B,C\right)  }{Q\left(
A,B\right)  }.
\]
Reflection in a line preserves quadrance between points and spread between
lines. Given three distinct points $A_{1},A_{2}$ and $A_{3}$, we use the
notation%
\[%
\begin{tabular}
[c]{lllll}%
$Q_{1}\equiv Q\left(  A_{2},A_{3}\right)  $ &  & $Q_{2}\equiv Q\left(
A_{1},A_{3}\right)  $ &  & $Q_{3}\equiv Q\left(  A_{1},A_{2}\right)  $%
\end{tabular}
\]
and%
\[%
\begin{tabular}
[c]{lllll}%
$s_{1}\equiv s\left(  A_{1}A_{2},A_{1}A_{3}\right)  $ &  & $s_{2}\equiv
s\left(  A_{2}A_{1},A_{2}A_{3}\right)  $ &  & $s_{3}\equiv s\left(  A_{3}%
A_{1},A_{3}A_{2}\right)  .$%
\end{tabular}
\]
A \textbf{triangle} $\overline{A_{1}A_{2}A_{3}}$ is a set of three
non-collinear points, and is \textbf{non-null} precisely when its three lines
$A_{1}A_{2},A_{2}A_{3}$ and $A_{1}A_{3}$ are non-null. Here are the five main
laws of rational trigonometry, which may be viewed as purely algebraic
identities involving only rational functions.

\begin{description}
\item[Triple quad formula] The points $A_{1},A_{2}$ and $A_{3}$ are collinear
precisely when%
\[
\left(  Q_{1}+Q_{2}+Q_{3}\right)  ^{2}=2\left(  Q_{1}^{2}+Q_{2}^{2}+Q_{3}%
^{2}\right)  .
\]

\item[Pythagoras' theorem] The lines $A_{1}A_{3}$ and $A_{2}A_{3}$ are
perpendicular precisely when%
\[
Q_{1}+Q_{2}=Q_{3}.
\]

\item[Spread law] For a non-null triangle $\overline{A_{1}A_{2}A_{3}}$
\[
\frac{s_{1}}{Q_{1}}=\frac{s_{2}}{Q_{2}}=\frac{s_{3}}{Q_{3}}.
\]

\item[Cross law] For a non-null triangle $\overline{A_{1}A_{2}A_{3}}$ define
the \textbf{cross} $c_{3}\equiv1-s_{3}.$ Then
\[
\left(  Q_{1}+Q_{2}-Q_{3}\right)  ^{2}=4Q_{1}Q_{2}c_{3}.
\]

\item[Triple spread formula] For a non-null triangle $\overline{A_{1}%
A_{2}A_{3}}$%
\[
\left(  s_{1}+s_{2}+s_{3}\right)  ^{2}=2\left(  s_{1}^{2}+s_{2}^{2}+s_{3}%
^{2}\right)  +4s_{1}s_{2}s_{3}.
\]

\end{description}

See \cite{Wild} for proofs, and many more facts about geometry in such a
purely algebraic setting.

\section*{Neuberg cubics}

Many interesting points, lines, circles, parabolas, hyperbolas and cubics have
been associated to a triangle in the plane, such as the centroid $G,$
orthocenter $O$, circumcenter $C,$ incenter $I,$ Euler line $e$ (which passes
through $O,G$ and $C$), nine-point circle and so on. Perhaps the most
remarkable of these is the \textit{Neuberg cubic}, which in the earlier
literature was called the $32$ point cubic, but these days is known to pass
through many more triangle centers (see \cite{EG}, \cite{Kimberling1},
\cite{Kimberling2}). Most such objects depend crucially on a metrical
structure on the affine plane.

Figure 1 shows the Neuberg cubic for the triangle $\overline{A_{1}A_{2}A_{3}}$
with vertices $A_{1}=\left[  0,0\right]  $, $A_{2}=\left[  1,0\right]  $ and
$A_{3}=\left[  3/4,3/4\right]  $. For this triangle the Euler line, which
passes through the orthocenter $O$ and the circumcenter $C,$ is horizontal.
Various incenters $I_{i}$ are shown, as well as reflections of the vertices in
the sides. These are just a few of the many points on the Neuberg cubic. Note
that the tangents to the four incenters are also horizontal, and it turns out
that the asymptote of the cubic is also.%
\begin{figure}
[h]
\begin{center}
\includegraphics[
height=6.0385cm,
width=10.9465cm
]%
{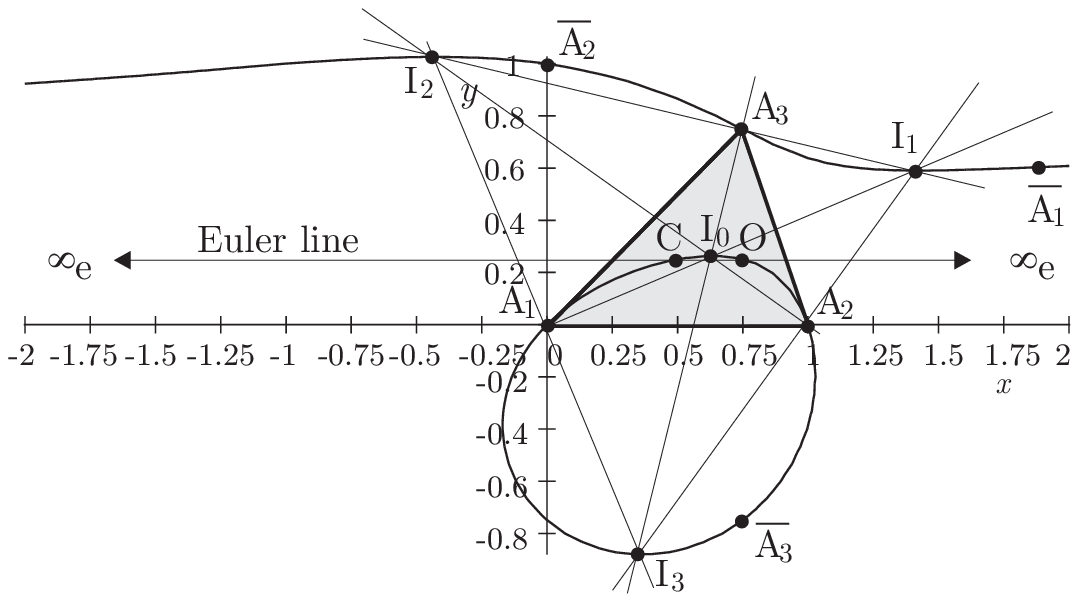}%
\caption{Neuberg cubic for $A_{1}=\left[  0,0\right]  $, $A_{2}=\left[
1,0\right]  $ and $A_{3}=\left[  3/4,3/4\right]  $}%
\label{NeubergReal}%
\end{center}
\end{figure}

With the completely algebraic language of rational trigonometry, we consider
such a picture for \textit{triangles in finite fields}. For this it will be
convenient to alter the usual point of view somewhat, elevating the quadrangle
$\overline{I_{0}I_{1}I_{2}I_{3}}$ of \textit{incenters} to a primary position.
This ensures that the reference triangle $\overline{A_{1}A_{2}A_{3}}$ actually
has vertex bisectors.

For a triangle $\overline{I_{1}I_{2}I_{3}}$ the altitudes from each point to
the opposite side intersect at the orthocenter, which we here denote $I_{0}$.
In terms of Cartesian coordinates $I_{j}=\left[  x_{j,}y_{j}\right]  $,
\[
x_{0}\equiv\frac{\left(
\begin{array}
[c]{c}%
x_{1}x_{2}y_{2}-x_{1}x_{3}y_{3}+x_{2}x_{3}y_{3}-x_{3}x_{2}y_{2}+x_{3}%
x_{1}y_{1}-x_{2}x_{1}y_{1}\\
+y_{1}y_{2}^{2}-y_{1}y_{3}^{2}+y_{2}y_{3}^{2}-y_{3}y_{2}^{2}+y_{3}y_{1}%
^{2}-y_{2}y_{1}^{2}%
\end{array}
\right)  }{\allowbreak x_{1}y_{2}-x_{1}y_{3}+x_{2}y_{3}-x_{3}y_{2}+x_{3}%
y_{1}-x_{2}y_{1}}%
\]
and%
\[
y_{0}\equiv\frac{\left(
\begin{array}
[c]{c}%
x_{1}y_{1}y_{2}-x_{1}y_{1}y_{3}+x_{2}y_{2}y_{3}-x_{3}y_{3}y_{2}+x_{3}%
y_{3}y_{1}-x_{2}y_{2}y_{1}\\
+x_{1}^{2}x_{2}-x_{1}^{2}x_{3}+x_{2}^{2}x_{3}-x_{3}^{2}x_{2}+x_{3}^{2}%
x_{1}-x_{2}^{2}x_{1}%
\end{array}
\right)  }{\allowbreak x_{1}y_{2}-x_{1}y_{3}+x_{2}y_{3}-x_{3}y_{2}+x_{3}%
y_{1}-x_{2}y_{1}}.
\]
In terms of barycentric coordinates, if the quadrances of the triangle
$\overline{I_{1}I_{2}I_{3}}$ are $R_{1},R_{2}$ and $R_{3}$ and
\begin{align*}
\mathcal{A}  &  =\left(  R_{1}+R_{2}+R_{3}\right)  ^{2}-2\left(  R_{1}%
^{2}+R_{2}^{2}+R_{3}^{3}\right) \\
&  =4\left(  x_{1}y_{2}-x_{1}y_{3}+x_{2}y_{3}-x_{3}y_{2}+x_{3}y_{1}-x_{2}%
y_{1}\right)  ^{2}%
\end{align*}
is the \textbf{quadrea} of the triangle (sixteen times the square of the area
in the decimal number situation), then $I_{0}=\beta_{1}I_{1}+\beta_{2}%
I_{2}+\beta_{3}I_{3}$ where
\begin{align*}
\beta_{1}  &  \equiv\left(  R_{3}+R_{1}-R_{2}\right)  \left(  R_{1}%
+R_{2}-R_{3}\right)  /\mathcal{A}\\
\beta_{2}  &  \equiv\left(  R_{1}+R_{2}-R_{3}\right)  \left(  R_{2}%
+R_{3}-R_{1}\right)  /\mathcal{A}\\
\beta_{3}  &  \equiv\left(  R_{2}+R_{3}-R_{1}\right)  \left(  R_{3}%
+R_{1}-R_{2}\right)  /\mathcal{A}.
\end{align*}

If $\overline{I_{1}I_{2}I_{3}}$ is non-null, which we henceforth assume, then
the feet of its altitudes exist and we call them respectively $A_{1},A_{2}$
and $A_{3}.$ Thus for example $I_{1},I_{0}$ and $A_{1}$ are collinear points
which lie on a line perpendicular to $I_{2}I_{3}.$ In the quadrangle
$\overline{I_{1}I_{2}I_{3}I_{4}}$ we have a complete symmetry between the four
points $I_{0},I_{1},I_{2}$ and $I_{3}.$ So we could have started with any
three of these points, and the orthocenter of such a triangle would have been
the fourth point, with the \textbf{orthic triangle} $\overline{A_{1}A_{2}%
A_{3}}$ obtained always the same.

\begin{theorem}
[Orthic triangle]The lines $I_{0}I_{1}$ and $I_{2}I_{3}$ are bisectors of the
vertex of $\overline{A_{1}A_{2}A_{3}}$ at $A_{1},$ in the sense that
\begin{align*}
s\left(  I_{0}I_{1},A_{1}A_{2}\right)   &  =s\left(  I_{0}I_{1},A_{1}%
A_{3}\right) \\
s\left(  I_{2}I_{3},A_{1}A_{2}\right)   &  =s\left(  I_{2}I_{3},A_{1}%
A_{3}\right)  .
\end{align*}

\end{theorem}

The proof uses a computer, and one can further verify that the former spread
is%
\[
\frac{\left(  x_{2}x_{3}-x_{1}x_{3}-x_{1}x_{2}-y_{1}y_{2}-y_{1}y_{3}%
+y_{2}y_{3}+x_{1}^{2}+y_{1}^{2}\right)  ^{2}}{R\allowbreak_{2}R_{3}}%
\]
while the latter spread is%
\[
\frac{\left(  x_{1}y_{2}-x_{1}y_{3}+x_{2}y_{3}-x_{3}y_{2}+x_{3}y_{1}%
-x_{2}y_{1}+x_{3}y_{2}\right)  ^{2}}{R\allowbreak_{2}R_{3}}.
\]
These two spreads sum to $1,$ as they must since $I_{0}I_{1}$ and $I_{2}I_{3}
$ are perpendicular. So the triangle $\overline{A_{1}A_{2}A_{3}}$ has a
special property: each of its vertices has bisectors. Over the decimal
numbers, every triangle has vertex bisectors, but in \cite{Wild} it is shown
that in general this amounts to the condition that the spreads of the triangle
are \textit{squares}.

For a point $P,$ let $P_{1},P_{2}$ and $P_{3}$ denote its reflections in the
sides $A_{2}A_{3},$ $A_{1}A_{3}$ and $A_{1}A_{2}$ respectively, and define the
\textbf{Neuberg cubic} $N_{c}$ of $\overline{A_{1}A_{2}A_{3}}$ to be the locus
of points $P$ such that $P_{1},P_{2}$ and $P_{3}$ are perspective with
$A_{1},A_{2}$ and $A_{3}$ respectively: in other words that $P_{1}A_{1}%
,P_{2}A_{2}$ and $P_{3}A_{3}$ are concurrent lines.

\section*{An example over $\mathbb{F}_{23}$}

We work in the prime field $\mathbb{F}_{23}$, in which the squares are
$1,4,9,16,2,13,3,18,12,8$ and $6.$ Note that $-1$ is not a square, but that
$3=7^{2}$ is a square. The latter fact implies that equilateral triangles
exist in $\mathbb{F}_{23}^{2}$. Let%
\[%
\begin{tabular}
[c]{lllll}%
$I_{1}=\left[  6,4\right]  $ &  & $I_{2}=\left[  22,22\right]  $ &  &
$I_{3}=\left[  21,12\right]  .$%
\end{tabular}
\]
These points have been chosen so that the orthocenter of $\overline{I_{1}%
I_{2}I_{3}}$ is $I_{0}=\left[  0,0\right]  $. The feet of the altitudes are%
\[%
\begin{tabular}
[c]{lllll}%
$A_{1}=\left[  13,1\right]  $ &  & $A_{2}=\left[  5,5\right]  $ &  &
$A_{3}=\left[  2,11\right]  .$%
\end{tabular}
\]
The lines of $\overline{A_{1}A_{2}A_{3}}$ are%
\[%
\begin{tabular}
[c]{lllll}%
$A_{1}A_{2}=\left\langle 3:6:1\right\rangle $ &  & $A_{2}A_{3}=\left\langle
6:3:1\right\rangle $ &  & $A_{1}A_{2}=\left\langle 12:4:1\right\rangle $%
\end{tabular}
\]
and the spreads of the triangle $\overline{A_{1}A_{2}A_{3}}$ are
\[%
\begin{tabular}
[c]{lllll}%
$s_{1}=12$ &  & $s_{2}=16$ &  & $s_{3}=6.$%
\end{tabular}
\]
Note that as expected these numbers are squares, and one can check that
\begin{align*}
s\left(  A_{1}I_{0},A_{1}A_{2}\right)   &  =s\left(  A_{1}I_{0},A_{1}%
A_{3}\right)  =5\\
s\left(  A_{2}I_{0},A_{2}A_{1}\right)   &  =s\left(  A_{2}I_{0},A_{2}%
A_{3}\right)  =-6\\
s\left(  A_{3}I_{0},A_{3}A_{1}\right)   &  =s\left(  A_{3}I_{0},A_{3}%
A_{2}\right)  =-7.
\end{align*}
The connection between for example the spread $s=5$ and the spread of its
`double' $r=12$ is given by the \textit{second spread polynomial},
\[
r=S_{2}\left(  s\right)  =4s\left(  1-s\right)
\]
which in chaos theory is known as the \textit{logistic map}. The spread
polynomials have many remarkable properties that hold also over finite fields,
see \cite{Wild}.

We need the following formula for a reflection.

\begin{theorem}
[Reflection of a point in a line]If $l\equiv\left\langle a:b:c\right\rangle $
is a non-null line and $A\equiv\left[  x,y\right]  $, then%
\[
\sigma_{l}\left(  A\right)  =\left[  \frac{\left(  b^{2}-a^{2}\right)
x-2aby-2ac}{a^{2}+b^{2}},\frac{-2abx+\left(  a^{2}-b^{2}\right)  y-2bc}%
{a^{2}+b^{2}}\right]  .
\]

\end{theorem}

Using this, the reflections of $P=\left[  x,y\right]  $ in the sides of
$\overline{A_{1}A_{2}A_{3}}$ are%
\begin{align*}
P_{1}  &  =\left[  4x+13y+12,13x+19y+6\right] \\
P_{2}  &  =\left[  13x+4y+1,4x+10y+8\right] \\
P_{3}  &  =\left[  19x+13y+6,13x+4y+12\right]  .
\end{align*}
The lines $P_{1}A_{1},P_{2}A_{2}$ and $P_{3}A_{3}$ are then%
\begin{align*}
&  \left\langle 13x+19y+5:19x+10y+1:19x+19y+3\right\rangle \\
&  \left\langle 4x+10y+3;10x+19y+4:22x+16y+11\right\rangle \\
&  \left\langle 13x+4y+1:4x+10y+19:22x+20y+19\right\rangle
\end{align*}
and these are concurrent precisely when%
\[%
\begin{vmatrix}
13x+19y+5 & 19x+10y+1 & 19x+19y+3\\
4x+10y+3 & 10x+19y+4 & 22x+16y+11\\
13x+4y+1 & 4x+10y+19 & 22x+20y+19
\end{vmatrix}
=0.
\]
Expanding gives the Neuberg cubic $\overline{A_{1}A_{2}A_{3}}:$ an affine
curve over $\mathbb{F}_{23}$ with equation%
\begin{equation}
y^{3}+x^{2}y+22y^{2}+7xy+9x^{2}+13y=0. \label{Neuberg}%
\end{equation}
The tangent line to a point $\left[  a,b\right]  $ on the curve has equation%
\[
x\left(  18a+7b+2ab\right)  +\allowbreak y\left(  7a+21b+a^{2}+3b^{2}%
+13\right)  +3b+7ab+9a^{2}+22b^{2}=0.
\]

There is another revealing way to obtain the Neuberg cubic.

\begin{theorem}
[Reflection of a line in a line]The reflection in the non-null line
$l\equiv\left\langle a:b:c\right\rangle $ sends $\left\langle a_{1}%
:b_{1}:c_{1}\right\rangle $ to%
\[
\left\langle \left(  a^{2}-b^{2}\right)  a_{1}+2abb_{1}:2aba_{1}-\left(
a^{2}-b^{2}\right)  b_{1}:2aca_{1}+2bcb_{1}-\left(  a^{2}+b^{2}\right)
c_{1}\right\rangle .
\]

\end{theorem}

In a triangle with vertex bisectors, the reflection of a line through a given
vertex in either of the vertex bisectors at that vertex is the same.

\begin{theorem}
[Isogonal conjugates]If a triangle $\overline{A_{1}A_{2}A_{3}}$ has vertex
bisectors at each vertex, then for any point $P$ the reflections of $A_{1}P,$
$A_{2}P$ and $A_{3}P$ in the vertex bisectors at $A_{1},A_{2}$ and $A_{3}$
respectively are concurrent.
\end{theorem}

The point of concurrence of these lines is $P^{\ast},$ the \textbf{isogonal
conjugate} of $P=\left[  x,y\right]  .$ The proof again is a calculation using
coordinates. We may use the reflection of a line in a line theorem to
establish a precise formula for $P^{\ast}$ in the special case of our example
reference triangle $\overline{A_{1}A_{2}A_{3}}$:%
\[
P^{\ast}=\left[  \frac{2x+22xy+2x^{2}+17y^{2}}{4x+20y+5x^{2}+5y^{2}+21}%
,\frac{\allowbreak2y+15xy+x^{2}+2y^{2}}{4x+20y+5x^{2}+5y^{2}+21}\right]  .
\]

Over the decimal numbers, the Neuberg cubic is \textit{also }the locus of
those $P=\left[  x,y\right]  $ such that $PP^{\ast}$ is parallel to the Euler
line. We can verify this also in our finite example, since this condition
amounts to%
\[
y=\frac{2y+15xy+x^{2}+2y^{2}}{4x+20y+5x^{2}+5y^{2}+21}%
\]
which in turn is equivalent to the equation (\ref{Neuberg}) of $N_{c}.$ It
follows that if $P$ lies on $N_{c},$ then so does $P^{\ast}$---this is a
useful way to obtain new points from old ones. In fact the Euler line is
parallel to the tangent to $N_{c}$ at the infinite point $\infty_{e}$.

The cubic (\ref{Neuberg}) is nonsingular, has $27$ points lying on it, and its
projective extension has one more point at infinity, namely $\infty
_{e}=\left[  1:0:0\right]  $. Here are all the points, the notation will be
explained more fully below:%
\[%
\begin{tabular}
[c]{llll}%
$\left[  0,0\right]  =I_{0}$ & $\left[  0,8\right]  =E_{3}^{\prime}$ &
$\left[  0,16\right]  =S^{\prime}$ & $\left[  2,11\right]  =A_{3}$\\
$\left[  3,13\right]  =S=\overline{A_{3}}$ & $\left[  4,5\right]  $ & $\left[
5,5\right]  =A_{2}$ & $\left[  5,14\right]  =\infty_{e}^{\ast}$\\
$\left[  6,4\right]  =I_{1}$ & $\left[  7,1\right]  $ & $\left[  7,2\right]
=E_{2}^{\prime}$ & $\left[  7,21\right]  =O$\\
$\left[  8,10\right]  =\overline{A_{1}}=E_{2}$ & $\left[  13,1\right]  =A_{1}$
& $\left[  13,7\right]  $ & $\left[  13,16\right]  =F^{\prime}$\\
$\left[  14,9\right]  $ & $\left[  16,11\right]  $ & $\left[  17,7\right]
=E_{1}^{\prime}$ & $\left[  17,8\right]  $\\
$\left[  17,9\right]  =\overline{A_{2}}=E_{1}$ & $\left[  18,21\right]
=C=E_{3}$ & $\left[  19,13\right]  =F$ & $\left[  21,2\right]  $\\
$\left[  21,10\right]  $ & $\left[  21,12\right]  =I_{3}$ & $\left[
22,22\right]  =I_{2}$ & $\left[  1:0:0\right]  =\infty_{e}$%
\end{tabular}
\]
%

\begin{figure}
[h]
\begin{center}
\includegraphics[
height=9.9113cm,
width=11.1401cm
]%
{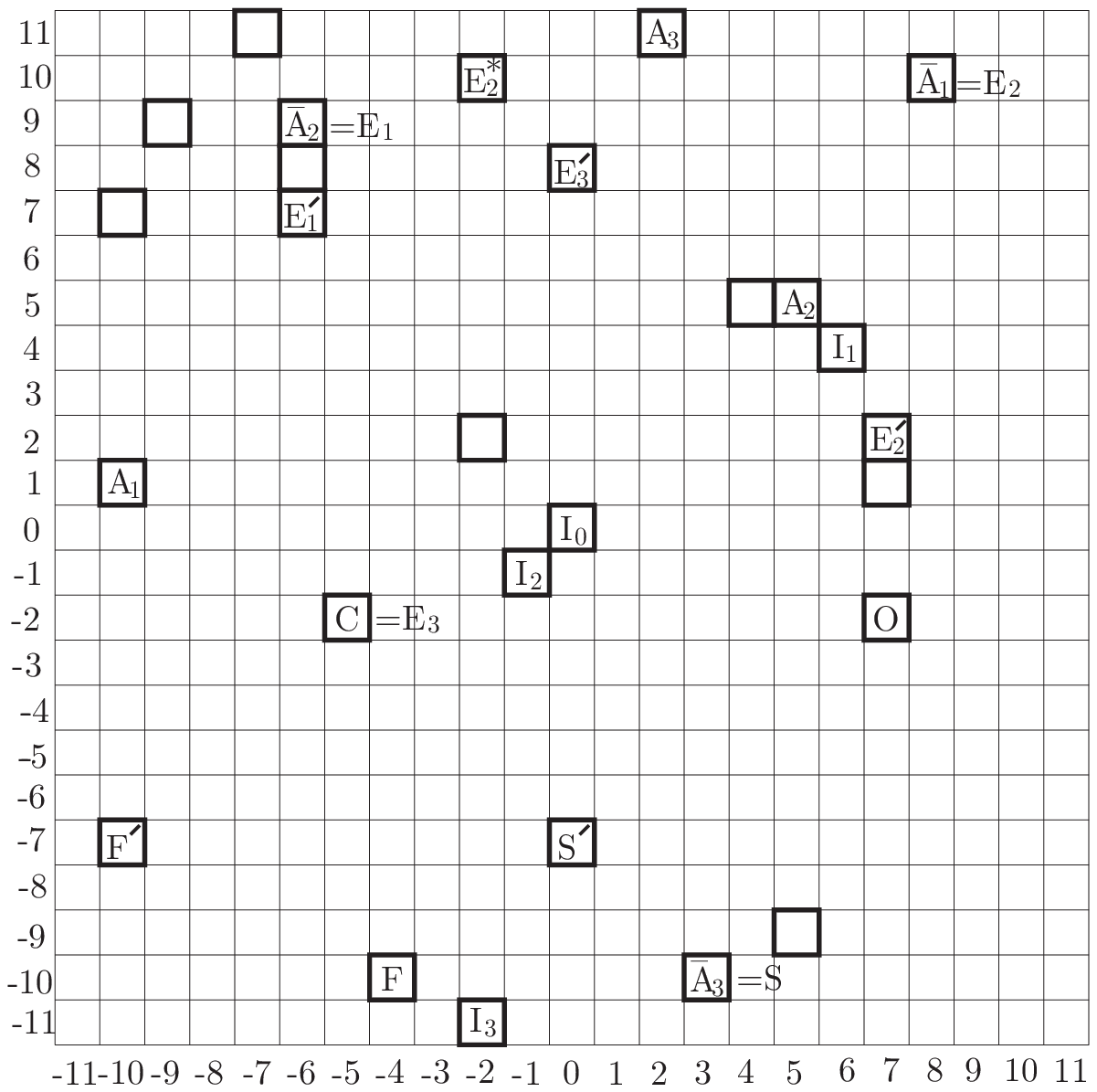}%
\caption{The Neuberg cubic for $A_{1}=\left[  13,1\right]  $, $A_{2}=\left[
5,5,\right]  $ and $A_{3}=\left[  2,11\right]  $}%
\label{NeubergCubic}%
\end{center}
\end{figure}

To define the group structure on the cubic, define $X\star Y$ to be the
(third) intersection of the line $XY$ with the (projective) cubic, so that for
example
\[
I_{0}\star I_{0}=\infty_{e}.
\]
We choose the base point of the group structure to be the point $I_{0}.$ Then
define
\[
X\cdot Y=\left(  X\star Y\right)  \star I_{0}%
\]
with inverse
\[
X^{-1}=X\star\left(  I_{0}\star I_{0}\right)  =X\star\infty_{e}=X^{\ast}.
\]
So we are writing the group multiplicatively, and note that $I_{0}$ is not a
flex, so that $X,Y$ and $Z$ collinear is equivalent to $X\cdot Y\cdot
Z=\infty_{e}$, not $X\cdot Y\cdot Z=I_{0}.$ Furthermore $X$ is of order two
when $X\star X=\infty_{e}$ and $X$ is not $I_{0},$ which happens when $X$ is
$I_{1},I_{2}$ or $I_{3},$ and the four incenters form a Klein $4$-group.

The triangle $\overline{A_{1}A_{2}A_{3}}$ can be recovered from the Neuberg
cubic by first finding the asymptote (tangent to the point at infinity), then
finding the four points $\overline{I_{0}I_{1}I_{2}I_{3}}$ on the cubic with a
tangent parallel to this asymptote, and then taking the orthic triangle of any
three of them. This can all be done algebraically using the group law, since
$I_{j}\star I_{j}=\infty_{e}$ and $A_{1}=I_{2}\star I_{3}$ etc.

\section*{Points on the Neuberg cubic}

Not only is the Neuberg cubic defined metrically, but it also has many points
on it that are metrical in nature. More than a hundred are known, we will
illustrate some of these for our example. The Neuberg cubic $N_{c}$ of
$\overline{A_{1}A_{2}A_{3}}$ first of all passes through $A_{1},A_{2}$ and
$A_{3}.$ It also passes through the reflections of these points in the sides
of the triangle, in this case
\[%
\begin{tabular}
[c]{lllll}%
$\overline{A_{1}}=\left[  8,10\right]  $ &  & $\overline{A_{2}}=\left[
17,9\right]  $ &  & $\overline{A_{3}}=\left[  3,13\right]  .$%
\end{tabular}
\]

It also passes through the four incenters $I_{0},I_{1},I_{2}$ and $I_{3}$ of
$\overline{A_{1}A_{2}A_{3}}.$ In a general field there is no notion of
`interior point', so these four incenters should be regarded symmetrically.
The Neuberg cubic passes through the orthocenter $O=\left[  7,21\right]  $ and
the circumcenter $C=\left[  18,21\right]  $ of $\overline{A_{1}A_{2}A_{3}},$
and these are isogonal conjugates, that is
\[
O^{\ast}=C.
\]
The line $OC$ is the Euler line $e$ of $\overline{A_{1}A_{2}A_{3}}$ and it has
equation $y=21,$ so that it is horizontal and passes through the infinite
point $\infty_{e}=\left[  1:0:0\right]  .$ Note that
\[
\infty_{e}^{\ast}=\left[  5,14\right]  .
\]

Since $3=7^{2}$ is a square, on any side of $\overline{A_{1}A_{2}A_{3}}$ we
may create two equilateral triangles, for example on the side containing
$A_{1}=\left[  13,1\right]  $ and $A_{3}=\left[  2,11\right]  $ we can choose
a third point%
\[
\left[  13+2,1+11\right]  /2\pm\frac{7}{2}\left[  1-11,2-13\right]
\]
namely%
\[%
\begin{tabular}
[c]{lllll}%
$E_{2}=\left[  8,10\right]  $ &  & $\mathrm{or}$ &  & $E_{2}^{\prime}=\left[
7,2\right]  .$%
\end{tabular}
\]
Thus $\overline{A_{1}A_{3}E_{2}}$ and $\overline{A_{1}A_{3}E_{2}^{\prime}}$
are equilateral triangles with%
\[
Q\left(  A_{1},A_{3}\right)  =Q\left(  A_{1},E_{2}\right)  =Q\left(
A_{3},E_{2}\right)  =Q\left(  A_{1},E_{2}^{\prime}\right)  =Q\left(
A_{3},E_{2}^{\prime}\right)  =14.
\]

As opposed to the case over the decimal numbers, there seems to be no obvious
notion of these triangles being either `exterior' or `interior' to
$\overline{A_{1}A_{2}A_{3}}.$ Using all three sides gives the six points%
\[%
\begin{tabular}
[c]{lllll}%
$E_{1}=\left[  17,9\right]  $ &  & $E_{2}=\left[  8,10\right]  $ &  &
$E_{3}=\left[  18,21\right]  $\\
$E_{1}^{\prime}=\left[  13,7\right]  $ &  & $E_{2}^{\prime}=\left[
7,2\right]  $ &  & $E_{3}^{\prime}=\left[  0,8\right]  $%
\end{tabular}
\]
and their isogonal conjugates
\[%
\begin{tabular}
[c]{lllll}%
$E_{1}^{\ast}=\left[  14,9\right]  $ &  & $E_{2}^{\ast}=\left[  21,10\right]
$ &  & $E_{3}^{\ast}=\left[  7,21\right]  $\\
$E_{1}^{\prime\ast}=\left[  17,7\right]  $ &  & $E_{2}^{\prime\ast}=\left[
21,2\right]  $ &  & $E_{3}^{\prime\ast}=\left[  17,8\right]  .$%
\end{tabular}
\]
All twelve of these points lie on the Neuberg cubic. Yet the centroids of the
six equilateral triangles thus formed are
\[%
\begin{tabular}
[c]{lllll}%
$G_{1}=\left[  8,16\right]  $ &  & $G_{2}=\left[  0,15\right]  $ &  &
$G_{3}=\left[  12,9\right]  $\\
$G_{1}^{\prime}=\left[  22,0\right]  $ &  & $G_{2}^{\prime}=\left[
15,20\right]  $ &  & $G_{3}^{\prime}=\left[  6,20\right]  $%
\end{tabular}
\]
and you may check that
\begin{align*}
Q\left(  G_{1},G_{2}\right)   &  =Q\left(  G_{2},G_{3}\right)  =Q\left(
G_{1},G_{3}\right)  =19\\
Q\left(  G_{1}^{\prime},G_{2}^{\prime}\right)   &  =Q\left(  G_{2}^{\prime
},G_{3}^{\prime}\right)  =Q\left(  G_{1}^{\prime},G_{3}^{\prime}\right)  =12
\end{align*}
so that Napolean's theorem that the centroids of both `external' and
`internal' equilateral triangles themselves form an equilateral triangle seems
to hold. It seems curious that the six points $E_{i}$ and $E_{j}^{\prime}$ are
thereby divided naturally into two groups.

The Fermat points of a triangle may be defined over the decimal numbers as the
perspectors of the `external and internal equilateral triangles', and with the
above interpretation, these points exist also in this field. There is another
approach to their definition. The vertex bisectors at $A_{1}$ of
$\overline{A_{1}A_{2}A_{3}}$ intersect $A_{2}A_{3}$ at the points
$X_{1}=\left[  20,21\right]  $ and $Y_{1}=\left[  12,14\right]  $. The circle
through these points with center the midpoint of $\overline{X_{1}Y_{1}}$ has
equation $\left(  x-16\right)  ^{2}+\left(  y-6\right)  ^{2}=11$ and is called
an \textbf{Apollonius circle} of $\overline{A_{1}A_{2}A_{3}}.$ There is also
such a circle starting with $A_{2},$ with equation $\left(  x-1\right)
^{2}+\left(  y-14\right)  ^{2}=5$ and one starting with $A_{3},$ with equation
$\left(  x-4\right)  ^{2}+\left(  y-17\right)  ^{2}=17.$ These three
Appollonius circles intersect at two points, called the \textbf{isodynamic
points }of $\overline{A_{1}A_{2}A_{3}},$ given by%
\[%
\begin{tabular}
[c]{lllll}%
$S=\left[  3,13\right]  $ &  & \textrm{and} &  & $S^{\prime}=\left[
0,16\right]  .$%
\end{tabular}
\]
The Neuberg cubic passes through the two isodynamic points. The centres of the
three Appollonius circles are collinear, and lie on the \textbf{Lemoine line
}with equation $16x+7y+1=0.$

The isogonal conjugates of the isodynamic points are the \textbf{Fermat
points}%
\[%
\begin{tabular}
[c]{lllll}%
$F=S^{\ast}=\left[  19,13\right]  $ &  & \textrm{and} &  & $F^{\prime}=\left(
S^{\prime}\right)  ^{\ast}=\left[  13,16\right]  .$%
\end{tabular}
\]
It may be checked that $F$ is also the centre of perspectivity between
$\overline{A_{1}A_{2}A_{3}}$ and $\overline{E_{1}E_{2}E_{3}},$ while
$F^{\prime}$ is the centre of perspectivity between $\overline{A_{1}A_{2}%
A_{3}}$ and $\overline{E_{1}^{\prime}E_{2}^{\prime}E_{3}^{\prime}}$. It may be
remarked that in the decimal number plane, the Fermat points also have an
interpretation in terms of minimizing the sum of the distances to the vertices
of the triangle, but this kind of statement cannot be expected to have a
simple analog in universal geometry.

The \textbf{Brocard line} with equation $10x+10y+1=0$ passes through the
circumcenter $C=\left[  18,21\right]  ,$ the \textbf{symmedian point}
$K=G^{\ast}=\left[  10,6\right]  $ and the two isodynamic points $S$ and
$S^{\prime}$. It is perpendicular to the Lemoine line.

\section*{Quadrangles and Desmic structure}

Elliptic curves naturally give rise to interesting configurations of $12$
points and $16$ lines, called by John Conway \textit{Desmic (or linking)}
\textit{structure, }where each line passes through three points and each point
lies on four lines (see \cite{Stothers}). To describe this situation, begin
with a triangle $ABC$ and two generic points $P$ and $Q.$ Then define
\[%
\begin{tabular}
[c]{lllll}%
$A^{\prime}=\left(  BP\right)  \left(  CQ\right)  $ &  & $B^{\prime}=\left(
CP\right)  \left(  AQ\right)  $ &  & $C^{\prime}=\left(  AP\right)  \left(
BQ\right)  $\\
$A^{\prime\prime}=\left(  BQ\right)  \left(  CP\right)  $ &  & $B^{\prime
\prime}=\left(  CQ\right)  \left(  AP\right)  $ &  & $C^{\prime\prime}=\left(
AQ\right)  \left(  BP\right)  $%
\end{tabular}
\]
This insures that $\overline{ABC}$ and $\overline{A^{\prime}B^{\prime
}C^{\prime}}$ are perspective from some perspector $D^{\prime\prime}$, that
$\overline{A^{\prime}B^{\prime}C^{\prime}}$ and $\overline{A^{\prime\prime
}B^{\prime\prime}C^{\prime\prime}}$ are perspective from some perspector $D$
and that $\overline{A^{\prime\prime}B^{\prime\prime}C^{\prime\prime}}$ and
$\overline{ABC}$ are perspective from some perspector $D^{\prime}.$
Furthermore the points $D,D^{\prime}$ and $D^{\prime\prime}$ are collinear.

Put another way, two triangles which are doubly perspective are triply
perspective (essentially a consequence of Pappus' theorem).
\begin{figure}
[h]
\begin{center}
\includegraphics[
height=6.2883cm,
width=12.8351cm
]%
{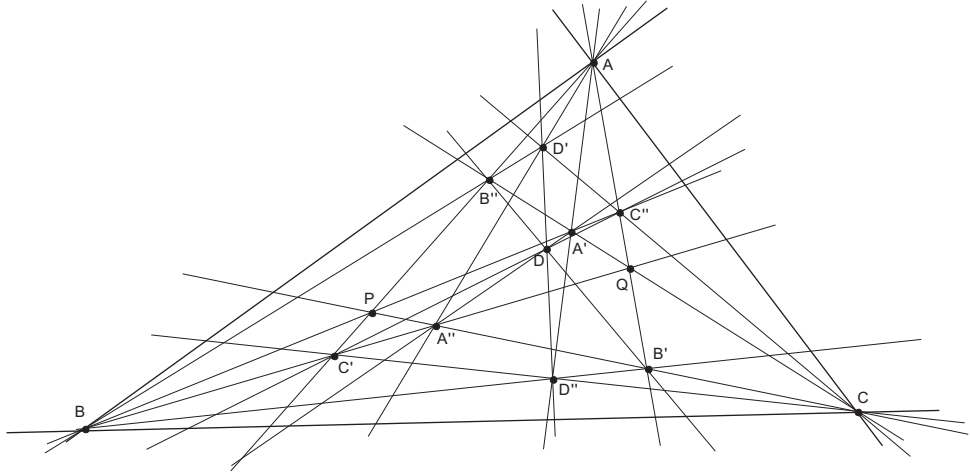}%
\caption{Desmic $16-12$ structure}%
\label{Desmic}%
\end{center}
\end{figure}
The various collinearities can be recorded in terms of the array%
\[%
\begin{tabular}
[c]{|l|l|l|l|}\hline
$A$ & $B$ & $C$ & $D$\\\hline
$A^{\prime}$ & $B^{\prime}$ & $C^{\prime}$ & $D^{\prime}$\\\hline
$A^{\prime\prime}$ & $B^{\prime\prime}$ & $C^{\prime\prime}$ & $D^{\prime
\prime}$\\\hline
\end{tabular}
\]
A triple of points from the array is collinear precisely when a) each is from
a different row, and b) if no $D$ is involved each is from a different column,
and c) if a $D$ is involved, then the other two are both from the same column.
An example of the former would be $A,B^{\prime\prime}$ and $C^{\prime},$ or
$C,B^{\prime}$ and $A^{\prime\prime}$. An example of the latter would be
$D^{\prime}$, $B$ and $B^{\prime\prime}$ or $D,D^{\prime} $ and $D^{\prime
\prime}$. There are then exactly $16$ such collinearities among these $12$ points.

Given a point $P$ on a cubic, there are in general four points $X_{1}%
,X_{2},X_{3}$ and $X_{4}$ on the cubic, other than $P,$ whose tangents pass
through $P.$ Call these four points a \textbf{quadrangle }of the Neuberg
cubic, and more specifically the \textbf{quadrangle to }$P.$ To illustrate
this, we write $X_{1},X_{2},X_{3},X_{4}:P.$

Here are some quadrangles for our cubic:%
\begin{align*}
A_{1},A_{2},A_{3},\infty_{e}  &  :\infty_{e}^{\ast}\\
I_{1},I_{2},I_{3},I_{o}  &  :\infty_{e}\\
\overline{A_{1}},\overline{A_{2}},\overline{A_{3}},C  &  :\left[  0,8\right]
\\
\overline{A_{1}}^{\ast},\overline{A_{2}}^{\ast},\overline{A_{3}}^{\ast},O  &
:\left[  17,8\right]
\end{align*}

Given three collinear points on a cubic, the associated quadrangles form a
Desmic structure. Here are some examples for our cubic%
\[%
\begin{tabular}
[c]{|l|l|l|l|}\hline
$A_{1}$ & $A_{2}$ & $A_{3}$ & $\infty_{e}$\\\hline
$I_{1}$ & $I_{2}$ & $I_{3}$ & $I_{0}$\\\hline
$I_{1}$ & $I_{2}$ & $I_{3}$ & $I_{0}$\\\hline
\end{tabular}
\]%
\[%
\begin{tabular}
[c]{|l|l|l|l|}\hline
$A_{1}$ & $A_{2}$ & $A_{3}$ & $\infty_{e}$\\\hline
$E_{1}$ & $E_{2}$ & $E_{3}$ & $\left[  3,13\right]  $\\\hline
$E_{1}^{\ast}$ & $E_{2}^{\ast}$ & $E_{3}^{\ast}$ & $\left[  19,13\right]
$\\\hline
\end{tabular}
\]%
\[%
\begin{tabular}
[c]{|l|l|l|l|}\hline
$A_{1}$ & $A_{2}$ & $A_{3}$ & $\infty_{e}$\\\hline
$E_{1}^{\prime}$ & $E_{2}^{\prime}$ & $E_{3}^{\prime}$ & $\left[  0,16\right]
$\\\hline
$E_{1}^{\prime\ast}$ & $E_{2}^{\prime\ast}$ & $E_{3}^{\prime\ast}$ & $\left[
13,16\right]  .$\\\hline
\end{tabular}
\]

We see that having recognized an (affine) cubic curve as a Neuberg cubic of a
triangle, we have lots of natural and deep geometry that connects to the group
multiplication. A natural question is: given an elliptic curve can we find an
affine view of it which is a Neuberg cubic? And if so, how can we classify
such views, and use them to understand elliptic curves?

Such an approach ought to be especially useful in the convenient laboratory
provided by finite fields.

\section*{Tangent conics to an affine cubic}

Here is a quite different use of affine coordinates in the study of an
elliptic curve. For more information and examples involving tangent conics,
see \cite{Wild}. Different metrical interpretations of tangent conics thus
allows one to distinguish points on an affine curve from the nature of the
tangent conic. Figure \ref{TangentConics} gives a view of some tangent conics
to $\left[  x_{0},y_{0}\right]  $ for the curve $y^{2}=x^{3}-x$ over the
decimal numbers. Tangent conics to points on the `egg' are ellipses, while
others are either hyperbolas opening horizontally, a pair of lines, or
hyperbolas opening vertically depending respectively on whether $x_{0}$ is
less than, equal to, or greater than $\sqrt{\frac{2}{3}\sqrt{3}+1}.$ Rather
remarkably, these tangent conics nowhere intersect.%
\begin{figure}
[h]
\begin{center}
\includegraphics[
height=5.4678cm,
width=5.4291cm
]%
{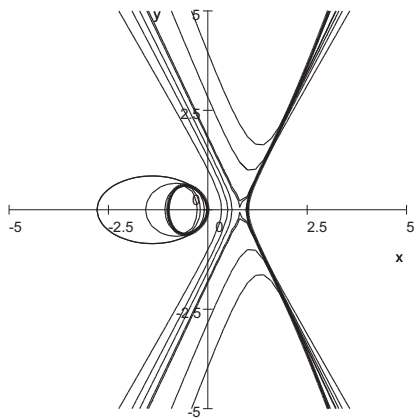}%
\caption{Tangent conics to $y^{2}=x^{3}-x$}%
\label{TangentConics}%
\end{center}
\end{figure}

\begin{theorem}
Over a field in which $-3$ is not a square, any two tangent conics of the
affine curve $y^{2}=ax^{3}+bx+c$ are either identical or disjoint.
\end{theorem}

\begin{proof}
Recall that the tangent conic to an affine curve is the part of the Taylor
expansion of degree two or less (see \cite[Chapter 19]{Wild}). More
specifically, to find the tangent conic to $y^{2}=ax^{3}+bx+c$ at a point
$A=\left[  x_{0},y_{0}\right]  $ on it, first translate the curve by $-A,$
yielding the equation%
\[
\left(  y+y_{0}\right)  ^{2}=a\left(  x+x_{0}\right)  ^{3}+b\left(
x+x_{0}\right)  +c
\]
or, after simplification using the fact that $A$ lies on the original curve,%
\[
y^{2}+2yy_{0}-ax^{3}-3ax^{2}\allowbreak x_{0}+x\left(  -b-3ax_{0}^{2}\right)
=0.
\]
Then take the quadratic part:
\[
y^{2}+2yy_{0}-3ax^{2}\allowbreak x_{0}+x\left(  -b-3ax_{0}^{2}\right)  =0
\]
and translate this conic back by $A,$ yielding%
\[
\left(  y-y_{0}\right)  ^{2}+2\left(  y-y_{0}\right)  y_{0}-3a\left(
x-x_{0}\right)  ^{2}x_{0}+\left(  x-x_{0}\right)  \left(  -b-3ax_{0}%
^{2}\right)  =0
\]
or after simplification%
\[
y^{2}-3ax^{2}x_{0}+x\left(  3ax_{0}^{2}-b\right)  -ax_{0}^{3}-c=0.
\]
To find the intersection between two such tangent conics
\begin{align*}
y^{2}-3ax^{2}x_{0}+x\left(  3ax_{0}^{2}-b\right)  -ax_{0}^{3}-c  &  =0\\
y^{2}-3ax^{2}x_{1}+x\left(  3ax_{1}^{2}-b\right)  -ax_{1}^{3}-c  &  =0
\end{align*}
take the difference between the two equations, which factors as%
\[
a\left(  x_{1}-x_{0}\right)  \left(  3x^{2}-3x\left(  x_{0}+x_{1}\right)
+x_{0}^{2}+x_{0}x_{1}+x_{1}^{2}\right)  .
\]
If $x_{0}=x_{1}$ then the tangent conics coincide. Otherwise we get an
intersection when the second quadratic factor has a zero. But its discriminant
is
\[
9\left(  x_{0}+x_{1}\right)  ^{2}-4\times3\left(  x_{0}^{2}+x_{0}x_{1}%
+x_{1}^{2}\right)  =\allowbreak\left(  -3\right)  \left(  x_{1}-x_{0}\right)
^{2}%
\]
and so if $-3$ is not a square then there is no solution and so the tangent
conics are disjoint.
\end{proof}

Note that the equation of the tangent conic%
\[
y^{2}-3ax^{2}x_{0}+x\left(  3ax_{0}^{2}-b\right)  -ax_{0}^{3}-c=0.
\]
can be rewritten as
\[
y^{2}-\left(  ax^{3}+bx+c\right)  +a\left(  x-x_{0}\right)  ^{3}=0.
\]

Figure 5 gives a view of some tangent conics for the curve $y^{2}=x^{3}+x.$
The tangent conic to $\left[  0,0\right]  $ is a parabola, and otherwise they
are either hyperbolas opening horizontally, a pair of lines, or hyperbolas
opening vertically depending respectively on whether $x_{0}$ is less than,
equal to, or greater than$\frac{1}{3}\sqrt{3}\sqrt{2\sqrt{3}-3}.$%
\begin{center}
\includegraphics[
height=5.3236cm,
width=5.2654cm
]%
{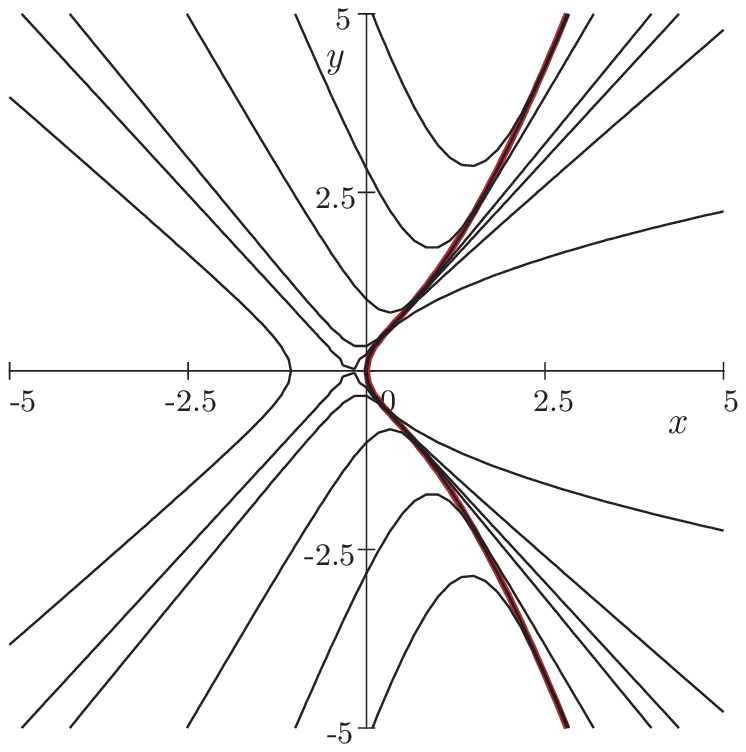}%
\\
Figure 5: Tangent conics to $y^{2}=x^{3}+x$%
\label{TangentConic3}%
\end{center}

Figure 6 shows the nodal cubic $y^{2}=x^{3}+x^{2}$ with tangent conic at
$\left[  x_{0},y_{0}\right]  $ given by%
\[
y^{2}+3xx_{0}^{2}+x^{2}\left(  -3x_{0}-1\right)  -x_{0}^{3}=0.
\]

We get a parabola when $x_{0}=-1/3,$ ellipses for $x_{0}$ less than that,
hyperbolas opening horizontally till $x_{0}=0,$ when we get the pair of lines
$y=\pm x,$ then hyperbolas opening upwards.%
\begin{center}
\includegraphics[
height=5.8038cm,
width=5.722cm
]%
{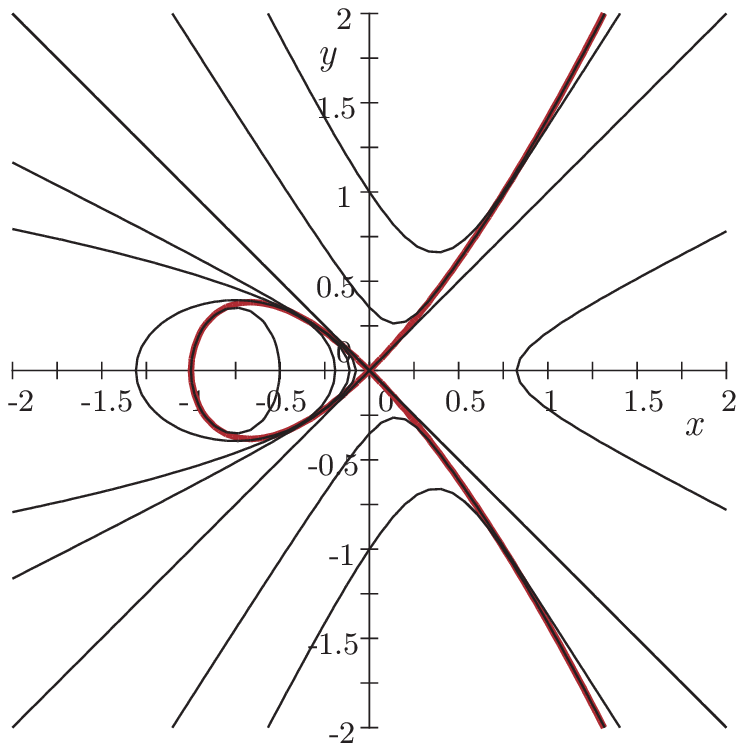}%
\\
Figure 6: Tangent conics to $y^{2}=x^{3}+x^{2}$%
\label{TangentConic4}%
\end{center}

\end{document}